\documentclass[12pt]{article}
\usepackage{mystyle}
\begin{document}
\title{The seed order}
\author{Gabriel Goldberg}
\maketitle
\begin{abstract}
We study various orders on countably complete ultrafilters on ordinals that coincide and are wellorders under a hypothesis called the Ultrapower Axiom. Our main focus is on the relationship between the Ultrapower Axiom and the linearity of these orders.
\end{abstract}
\section{Introduction}
In this paper, we study several related orders on countably complete ultrafilters that all coincide under a set theoretic hypothesis called the {\it Ultrapower Axiom} (UA). The first, called the seed order and denoted \(\swo\), was introduced in \cite{UA}, where the following theorem was established:
\begin{thm}[UA]
The seed order is a wellorder of the class of countably complete uniform ultrafilters on ordinals.
\end{thm}
Here we will show quite easily that in ZFC one cannot even establish the transitivity of \(\swo\), since the transitivity of \(\swo\) is equivalent to UA. We instead introduce another very natural partial order called the {\it Ketonen order} and denoted \(\sE\), a special case of which was briefly considered in \cite{Ketonen}. We then show that the Ketonen order yields the appropriate definition of the seed order in the context of ZFC:
\begin{thm}
The Ketonen seed order is a strict wellfounded partial order on the class of countably complete uniform ultrafilters on ordinals.
\end{thm}

The Ketonen order extends the usual seed order, so we can conclude (as a theorem of ZFC) that the usual seed order is a strict wellfounded relation. Moreover, under UA the Ketonen order and the seed order are equal and both are wellorders. 

The main theorem of this paper concerns the linearity of the Ketonen order:
\begin{thm}\label{Linearity}
The following are equivalent:
\begin{enumerate}[(1)]
\item The Ketonen order is linear.
\item The Ultrapower Axiom holds.
\end{enumerate}
\end{thm}

The last order we consider is the Lipschitz order on countably complete ultrafilters, which is a direct generalization of the Lipschitz order on subsets of \({}^{\omega}2\).
\begin{thm}[UA]
The Lipschitz order coincides with the seed order on countably complete uniform ultrafilters on an ordinal \(\delta\).
\end{thm}

\begin{cor}[UA]
Lipschitz Determinacy holds for countably complete ultrafilters on ordinals.
\end{cor}

Therefore the Ultrapower Axiom implies a form of long determinacy for ultrafilters. We also prove a partial converse, but the following remains open:
\begin{qst}
Suppose the Lipschitz order is linear on countably complete uniform ultrafilters on ordinals. Does the Ultrapower Axiom hold?
\end{qst}
\section{Notation}
We introduce a bit of notation that will save some ink in the sequel.

\begin{defn}
Suppose \(P\) is an inner model and \(U\) is an \(P\)-ultrafilter. We write \((M_U)^P\) to denote the ultrapower of \(P\) by \(U\) using functions in \(P\). We write \((j_U)^P\) to denote the ultrapower embedding from \(P\) to \((M_U)^P\) associated to \(U\). For any function \(f\in P\), we write \([f]^P_U\) to denote the point represented by \(f\) in \((M_U)^P\).
\end{defn}

We will often omit the parentheses in this notation, writing \(j_U^P\) and \(M_U^P\).

\begin{defn}
Suppose \(M\) and \(N\) are inner models. An elementary embedding \(i : M \to N\) is an {\it ultrapower embedding} if there is an \(M\)-ultrafilter \(U\) such that \(i = (j_U)^M\). An elementary embedding \(i : M \to N\) is an {\it internal ultrapower embedding} if there is an \(M\)-ultrafilter \(U\in M\) such that \(i = (j_U)^M\).

We say \(N\) is an {\it ultrapower} of \(M\) if there is an ultrapower embedding from \(M\) to \(N\). We say \(N\) is an {\it internal ultrapower} of \(M\) if there is an internal ultrapower embedding from \(M\) to \(N\). 
\end{defn}

Our definition of an ultrapower embedding reflects our focus on countably complete ultrafilters: {\it for example, an elementary embedding \(j : V\to M\) where \(M\) is illfounded is never an ultrapower embedding by our definition}. We note that there is a characterization of ultrapower embeddings that does not mention ultrafilters:
\begin{lma}
An elementary embedding \(j : M\to N\) is an ultrapower embedding if there is some \(a\in N\) such that \(N = H^N(j[M]\cup \{a\})\).\qed
\end{lma}

\section{The Ultrapower Axiom}
We find that the following notational device clarifies the statements of various definitions and theorems.

\begin{defn}
Suppose \(M_0\), \(M_1,\) and \(N\) are transitive models of set theory. We write \((i_0,i_1) : (M_0,M_1)\to N\) to denote that \(i_0 : M_0\to N\) and \(i_1 : M_1\to N\) are elementary embeddings.
\end{defn}

\begin{defn}
Suppose \(j_0 : V\to M_0\) and \(j_1 : V\to M_1\) are ultrapower embeddings. A {\it comparison} of \((j_0,j_1)\) is a pair \((i_0,i_1) : (M_0,M_1)\to N\) of internal ultrapower embeddings such that \(i_0 \circ j_0 = i_1\circ j_0\).
\end{defn}

\begin{ua}
Every pair of ultrapower embeddings admits a comparison.
\end{ua}

\section{The seed order}
\begin{defn}
If \(\alpha\) is an ordinal, the {\it tail filter} on \(\alpha\) is the filter generated by sets of the form \(\alpha\setminus \beta\) for \(\beta < \alpha\). 

An ultrafilter on \(\alpha\) is {\it tail uniform} (or just {\it uniform}) if it extends the tail filter. 

If \(U\) is a uniform ultrafilter on an ordinal, then the {\it space} of \(U\), denoted \(\textsc{sp}(U)\), is the unique ordinal \(\alpha\) such that \(\alpha\in U\). 

The class of countably complete uniform ultrafilters is denoted by \(\Un\).
\end{defn}

\begin{defn}
The {\it seed order} is defined on countably complete uniform ultrafilters \(U_0\) and \(U_1\) by setting \(U_0\swo U_1\) if there is a comparison \((i_0,i_1) : (M_{U_0},M_{U_1})\to N\) of \((j_{U_0},j_{U_1})\) such that \(i_0([\text{id}]_{U_0}) < i_1([\text{id}]_{U_1})\).
\end{defn}

In \cite{UA}, we prove that the Ultrapower Axiom implies that the seed order is linear on countably complete ultrafilters. The totality of the seed order on countably complete ultrafilters trivially implies the Ultrapower Axiom. Here we just note that the transitivity of the seed order also implies the Ultrapower Axiom. We need an easy lemma whose proof appears in \cite{UA}.

\begin{lma}\label{SpaceLemma1}
Suppose \(U,W\in \Un\) and \(U \swo W\). Then \(\textsc{sp}(U) \leq \textsc{sp}(W)\).
\end{lma}

\begin{prp}
Assume the seed order is transitive. Then the Ultrapower Axiom holds.
\begin{proof}[Sketch]
Fix countably complete ultrafilters \(U,W\). We will show that the pair \((j_U,j_W)\) admits a comparison. 

Let \(\alpha = \textsc{sp}(U)\). Let \(W'\) be the uniform ultrafilter derived from \(W\) using \(\langle [\text{id}]_W,j_W(\alpha)\rangle\) where \(\langle -,-\rangle\) denotes some reasonable pairing function. Then by \cref{SpaceLemma1}, \[U\swo P_\alpha\wo W'\] (Recall that \(P_\alpha\) denotes the uniform principal ultrafilter on \(\alpha+1\) concentrated at \(\alpha\).) By the transitivity of the seed order, \(U\swo W'\). Therefore there is a comparison of \((j_U,j_{W'})\). But \(j_{W'} = j_W\), so there is a comparison of \((j_U,j_W)\), as desired.
\end{proof}
\end{prp}

\subsection{The Ketonen order}
The fact that the transitivity of the seed order is equivalent to UA suggests that in the context of ZFC alone, the definition of the seed order is not really the correct one. This motivates the definition of the Ketonen order, which combinatorially is actually somewhat simpler than that of the seed order.

\begin{defn}
The {\it Ketonen order} is defined on countably complete uniform ultrafilters \(U\) and \(W\) by setting \(U\sE W\) if there is a sequence of countably complete ultrafilters \(U_\alpha\) on \(\alpha\), defined whenever \(0 < \alpha <\textsc{sp}(W)\), such that for all \(X\subseteq \textsc{sp}(U)\), \(X\in U\) if and only if \(X\cap \alpha\in U_\alpha\) for \(W\)-almost all \(\alpha\). 
\end{defn}

The main theorem of this section is that the Ketonen order is a strict wellfounded partial order on the class of countably complete uniform utlrafilters on ordinals. 

We begin by rephrasing the definition of the Ketonen order in two simple ways.
\begin{defn}
Suppose \(U\) is a countably complete ultrafilter and in \(M_U\), \(W'\) is a countably complete uniform ultrafilter on an ordinal \(\delta'\). Then the {\it \(U\)-limit of \(W'\)} is the ultrafilter \[U^-(W') = \{X\subseteq \delta : j_U(X)\cap \delta'\in W'\}\] where \(\delta\) is the least ordinal such that \(\delta'\leq j_U(\delta)\).
\end{defn}

Clearly \(U^-(W')\) is a countably complete uniform ultrafilter on \(\delta\), and \(U^-(W)\) is invariant under replacing \(U\) with an isomorphic ultrafilter. The Ketonen order is related to limits in the following straightforward way:

\begin{lma}[UA]\label{SeedLimit}
If \(U\) and \(W\) are countably complete uniform ultrafilters, then \(U\sE W\) if and only if there a countably complete uniform ultrafilter \(U'\) of \(M_W\) such that \(\textsc{sp}(U') \leq [\textnormal{id}]_W\) and \(W^-(U') = U\).\qed
\end{lma}

On the other hand, there is a characterization of limits in terms of elementary embeddings:

\begin{lma}\label{LimitEmbedding}
Suppose \(U\) is a countably complete ultrafilter and \(W'\) is a countably complete uniform ultrafilter of \(M_U\). Then \(U^-(W')\) is the unique uniform ultrafilter \(W\) such that there is an elementary embedding \(k : M_W\to M^{M_U}_{W'}\) such that \(k\circ j_W = j^{M_U}_{W'}\circ j_U\) and \(k([\textnormal{id}]_W)= [\textnormal{id}]_{W'}^{M_U}\).\qed
\end{lma}

This leads to a characterization of the Ketonen order that looks very similar to the definition of the seed order.

\begin{defn}
Suppose \(j_0 : V\to M_0\) and \(j_1 : V\to M_1\) are ultrapower embeddings. A pair of ultrapower embeddings \((i_0,i_1) :(M_0,M_1)\to N\) is a {\it semicomparison} of \((j_0,j_1)\) if \(i_1\) is an internal ultrapower embedding of \(M_1\) and \(i_0\circ j_0 = i_1 \circ j_1\).
\end{defn}

We warn that the notion of a semicomparison of \((j_0,j_1)\) is not symmetric in \(j_0\) and \(j_1\). 

\begin{lma}\label{SeedComparison}
If \(U_0\) and \(U_1\) are countably complete uniform ultrafilters, then \(U_0\sE U_1\) if and only if there is a semicomparison \((i_0,i_1) : (M_{U_0},M_{U_1})\to N\) of \((j_{U_0},j_{U_1})\) such that \(i_0([\textnormal{id}]_{U_0}) < i_1([\textnormal{id}]_{U_1})\).
\begin{proof}
Suppose \(U_0\sE U_1\). By \cref{SeedLimit}, fix a countably complete uniform ultrafilter \(W'\) of \(M_{U_1}\) such that \(\textsc{sp}(W') \leq [\text{id}]_{U_1}\) and \(U_1^-(W') = U_0\). By \cref{LimitEmbedding}, there is an elementary embedding \(k : M_{U_0}\to M^{M_{U_1}}_{W'}\) such that \(k\circ j_{U_0} = j_{W'}^{M_{U_1}}\circ j_{U_1}\) and \[k([\text{id}]_{U_0}) = [\text{id}]_{W'}^{M_{U_1}} < j_{W'}^{M_{U_1}}([\text{id}]_{U_1}) \] with the last inequality following from the fact that \(\textsc{sp}(W') \leq [\text{id}]_{U_1}\). It follows that \((k,j_{W'}^{M_{U_1}})\) is a semicomparison of \((j_{U_0},j_{U_1})\) with \(k([\text{id}]_{U_0}) < j_{W'}^{M_{U_1}}([\text{id}]_{U_1})\), as desired.

Conversely suppose \((i_0,i_1) : (M_{U_0},M_{U_1})\to N\) is a semicomparison of \((j_{U_0},j_{U_1})\) with \(i_0([\text{id}]_{U_0}) < i_1([\text{id}]_{U_1})\). Let \(W'\) be the uniform \(M_{U_1}\)-ultrafilter derived from \(i_1\) using \(i_0([\text{id}]_{U_0})\). Then \(\textsc{sp}(W)\leq [\text{id}]_{U_1}\). Moreover easily \(U_1^-(W') = U_0\). So by \cref{SeedLimit}, \(U_0\sE U_1\), as desired.
\end{proof}
\end{lma}

\begin{cor}\label{KSeed}
The Ketonen order extends the seed order.
\end{cor}

We use this characterization of the Ketonen order to prove its transitivity.
\begin{lma}
The Ketonen order is transitive.
\begin{proof}
Suppose \(U_0\sE U_1 \sE U_2\). We will show \(U_0\sE U_2\). Fix comparisons \((i_0,i_1) : (M_{U_0},M_{U_1})\to N\) and \((i_1',i'_2) : (M_{U_1},M_{U_2})\to N'\). Let \(P = i_1'(N)\). Then \[(i_1'\circ i_0, i_1'(i_1)\circ i'_2) : (M_{U_0},M_{U_2})\to P\] and \[i_1'\circ i_0\circ j_{U_0} = i_1'\circ i_1\circ j_{U_1} = i_1'(i_1)\circ i_1'\circ j_{U_1}= i_1'(i_1)\circ i'_2\circ j_{U_2}\] so \((i_1'\circ i_0, i_1'(i_1)\circ i'_2)\) is a semicomparison of \((j_{U_0},j_{U_1})\).

Finally \[i_1'\circ i_0([\text{id}]_{U_0}) < i_1'\circ i_1([\text{id}]_{U_1}) = i_1'(i_1)\circ i_1'([\text{id}]_{U_1}) < i_1'(i_1)\circ i_2'([\text{id}]_{U_2})\] so by \cref{SeedComparison}, \((i_1'\circ i_0, i_1'(i_1)\circ i'_2)\) witnesses \(U_0\sE U_2\), as desired.
\end{proof}
\end{lma}

The proof that the Ketonen order is wellfounded is a bit more subtle, and apparently it was not known to Ketonen (who proved it only in the special case of weakly normal ultrafilters). We give a combinatorial proof here that uses the following lemma which allows us to copy the structure of the Ketonen order in \(V\) into its ultrapowers.

\begin{lma}\label{Resemblance}
Suppose \(U,W,\) and \(Z\) are uniform ultrafilters and \(U\sE W\). Suppose \(W_*\in \textnormal{Un}^{M_Z}\) is such that \(Z^-(W_*) = W\). Then there is some \(U_*\sE^{M_Z} W_*\) with \(Z^-(U_*) = U\).
\begin{proof}
Since \(U\sE W\), by \cref{SeedLimit} there is  a uniform ultrafilter \(U'\) of \(M_W\) with \(\textsc{sp}(U') \leq [\text{id}]_W\) and \(W^-(U') = U\). 

Let \(k : M_W \to M^{M_Z}_{W_*}\) be the unique elementary embedding with \(k\circ j_W = j^{M_Z}_{W_*}\circ j_Z\) and \(k([\text{id}_W]) = [\text{id}]^{M_W}_{W*}\).
\begin{figure}[h]
\begin{center}
\includegraphics[scale=.8]{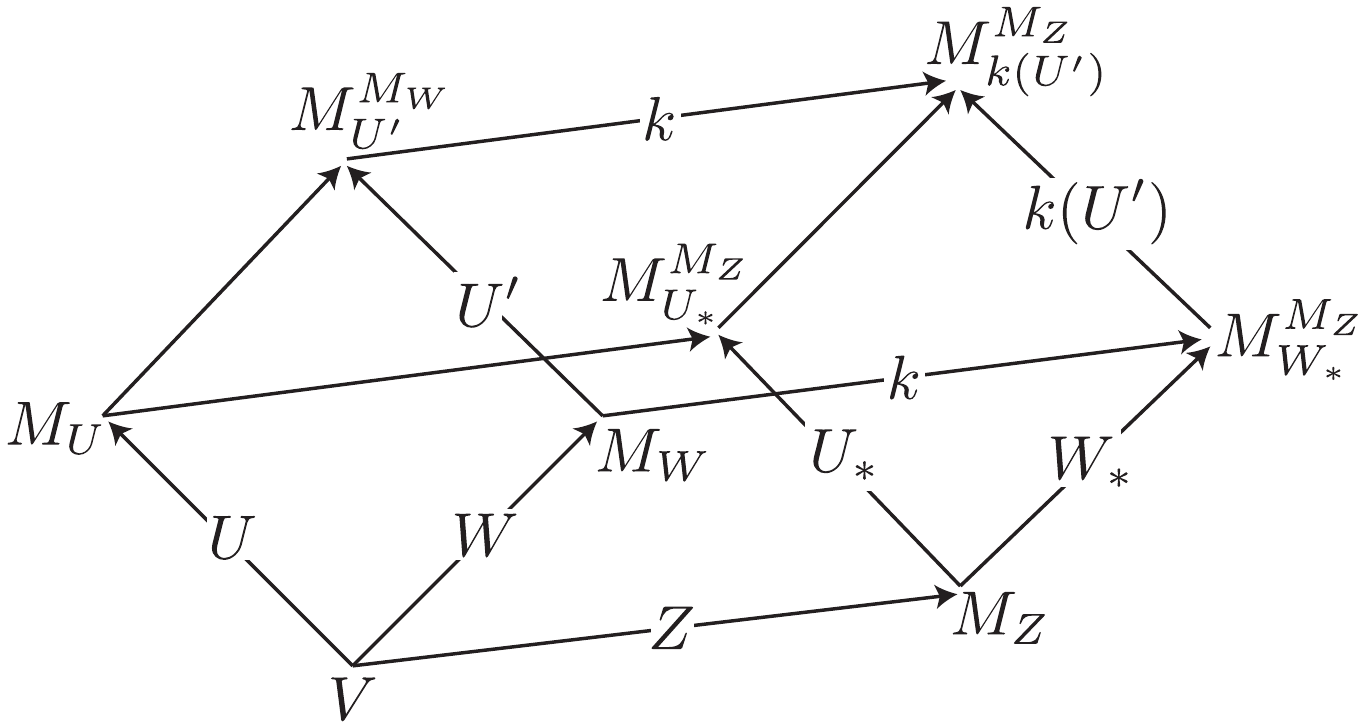}
\caption{Copying \(\sE\) into \(M_Z\).}
\end{center}
\end{figure}
Let \(U_* = W_*^-(k(U'))\). Since \(\textsc{sp}(k(U')) \leq k([\text{id}]_W) = [\text{id}]^{M_Z}_{W_*}\), \(k(U')\) witnesses that \(U_*\sE^{M_Z} W_*\). Moreover,
\begin{align*}Z^-(U_*) &= Z^-(W_*^-(k(U')))\\
&= \{X\subseteq \delta : j^{M_Z}_{W_*}\circ j_Z(X)\in k(U')\}\\
&= \{X\subseteq \delta : k\circ j_W(X)\in k(U')\}\\
&= \{X\subseteq \delta : j_W(X)\in U'\}\\
&= W^-(U') = U
\end{align*}
This completes the proof.
\end{proof}
\end{lma}
\begin{thm}\label{Wellfoundedness} The Ketonen order is wellfounded.
\begin{proof}
Assume towards a contradiction that \(\delta\) is the least ordinal carrying a countably complete uniform ultrafilter \(U_0\) below which the Ketonen order is illfounded. 

Fix a sequence \(U_0\sgE U_1 \sgE U_2 \sgE\cdots\) witnessing this. Using the fact that \(U_1\sE U_0\), fix \(U^*_1\) such that \(U_0^-(U^*_1) = U_1\) and \(\textsc{sp}(U^*_1) \leq [\text{id}]_{U_0}\). Assume \(n\geq 1\) and we have defined \(U^*_1\sgE^{M_{U_0}} U^*_2\sgE^{M_{U_0}} \cdots \sgE^{M_{U_0}} U^*_n\) such that \(U_0^-(U^*_n) = U_n\). Using \cref{Resemblance} and the fact that \(U_{n+1} \sE U_n\), fix \(U^*_{n+1}\sE^{M_{U_0}} U^*_n\) such that \(U^-_0(U^*_{n+1}) = U_{n+1}\).

Continuing this way, we produce a sequence \(U^*_1\sgE^{M_{U_0}} U^*_2\sgE^{M_{U_0}} U^*_3\sgE^{M_{U_0}} \cdots\) witnessing that the seed order of \(M_{U_0}\) is illfounded below \(U^*_1\). 

But \(\textsc{sp}(U^*_1) \leq [\text{id}]_{U_0} < j_{U_0}(\delta)\), while in \(M_{U_0}\), the least ordinal carrying a countably complete uniform ultrafilter below which the Ketonen order is illfounded is \(j_{U_0}(\delta)\) by elementarity. Wellfoundedness is absolute between \(M_{U_0}\) and \(V\) since \(U_0\) is countably complete. This is a contradiction.
\end{proof}
\end{thm}
\section{The Reciprocity Lemma}
In this section we prove that the linearity of the Ketonen order implies the Ultrapower Axiom. We begin by establishing some a general fact about limits ultrafilters under UA that motivates the proof.
\subsection{Translations and canonical comparisons}
In this section we define the notion of a translation function, which gives a seed order theoretic perspective on the structure of comparisons of ultrafilters.

\begin{defn}[UA]
For \(U,W\in \Un\), \(\tr{W}{U}\) denotes the \(\swo^{M_W}\)-least \(U'\in \Un^{M_W}\) such that \(W^-(U') = U\). The function \(t_W : \Un\to \Un^{M_W}\) is called the {\it translation function} associated to \(W\).
\end{defn}

An immediate consequence of the definition of translation functions is the following bound:

\begin{lma}[UA]\label{BoundingLemma}
For any \(U,W\in \Un\), \(\tr U W \wo^{M_U} j_U(W)\).
\begin{proof}
Note that \(U^-(j_U(W)) = W\), so by the minimality of \(\tr U W\), \(\tr U W \wo^{M_U} j_U(W)\).
\end{proof}
\end{lma}

The main fact about translation functions is that they arise from comparisons:

\begin{thm}[UA]\label{ReciprocityCor}
Suppose \(U,W\in \Un\). Suppose \((i'_U,i'_W) : (M_U,M_W)\to N'\) is a comparison of \((j_U,j_W)\). Then \(\tr{W}{U}\) is the \(M_W\)-ultrafilter derived from \(i'_W\) using \(i'_U([\textnormal{id}]_U)\).
\end{thm}

This will follow from immediately from the special case in which the comparison in question is the canonical comparison:
\begin{defn}
Suppose \(j_0 : V\to M_0\) and \(j_1 : V\to M_1\) are ultrapower embeddings. A comparison \((i_0,i_1) : (M_0,M_1)\to N\) of \((j_0,j_1)\) is {\it canonical} if for any comparison \((i_0',i_1') : (M_0,M_1)\to N'\), there is an elementary embedding \(h : N\to N'\) such that \(h\circ i_0 = i_0'\) and \(h\circ i_1 = i_1'\).
\end{defn}

The following theorem is proved in \cite{RF} Section 5:

\begin{thm}[UA]
Every pair of ultrapower embeddings has a canonical comparison.
\end{thm}

To prove \cref{ReciprocityCor}, we will first show:

\begin{lma}[UA]\label{Reciprocity1}
Suppose \(U,W\in \Un\). Suppose \((i_U,i_W) : (M_U,M_W)\to N\) is the canonical comparison of \((j_U,j_W)\). Then \(\tr{W}{U}\) is the \(M_W\)-ultrafilter derived from \(i_W\) using \(i_U([\textnormal{id}]_U)\).
\end{lma}

For the proof we need a key fact about definable elementary embeddings proved in \cite{IR}.

\begin{thm}\label{MinDefEmb}
Suppose \(M\) and \(N\) are inner models and \(j,i : M \to N\) are elementary embeddings. Assume \(j\) is definable over \(M\) from parameters. Then \(j(\alpha) \leq i(\alpha)\) for all ordinals \(\alpha\).
\end{thm}

\begin{proof}[Proof of \cref{Reciprocity1}]
Let \(U'\) be the uniform \(M_W\)-ultrafilter derived from \(i_W\) using \(i_U([\textnormal{id}]_U)\). Note that \((j_{U'})^{M_W} = i_W\), by the proof of \cref{MinimalUltrapowers}.
By (3) to (1) of \cref{LimitEmbedding} with \(k = i_U\), \(U = W^-(U')\).

Fix \(U''\in M_W\) such that \(W^-(U'') = U\). We must show that \(U'\wo^{M_W} U''\). Let \((\ell',\ell'') : (M_{U'}^{M_W},M_{U''}^{M_W})\to N\) be a comparison of \((j_{U'}^{M_W},j_{U''}^{M_W})\). We must show that \(\ell'([\text{id}]^{M_W}_{U'}) \leq \ell''([\text{id}]^{M_W}_{U''})\).
Using \cref{LimitEmbedding}, let \(k : M_U\to (M_{U''})^{M_W}\) be the elementary embedding with \(k\circ j_U = (j_{U''})^{M_W}\circ j_W\) and \(k([\text{id}]_U) = [\text{id}]^{M_W}_{U''}\). Then \(\ell'\circ i_U\) and \(\ell''\circ k\) are elementary embeddings from \(M_U\) to \(N\). Moreover \(\ell'\circ i_U\) is definable from parameters over \(M_U\). Therefore by \cref{MinDefEmb}, for all ordinals \(\alpha\), \(\ell'\circ i_U(\alpha) \leq \ell''\circ k(\alpha)\). In particular, \(\ell'\circ i_U([\text{id}]_U) \leq \ell''\circ k([\text{id}]_U)\). In other words, \(\ell'([\text{id}]_{U'}^{M_W})\leq \ell''([\text{id}]_{U''}^{M_W})\), as desired.
\end{proof}

\begin{proof}[Proof of \cref{ReciprocityCor}]
Let \((i_U,i_W) : (M_U,M_W)\to N\) be the canonical comparison of \((j_U,j_W)\) and let \(h : N \to N'\) be an elementary embedding such that \(i_U' = h\circ i_U\) and \(i_W'= h\circ i_W\). Then the uniform \(M_W\)-ultrafilter derived from \(i_W'\) using \(i'_U([\textnormal{id}]_U)\) is the same as the one derived from \(i_W\) using \(i_U([\text{id}]_U)\) since \begin{align*}
i'_U([\textnormal{id}]_U) \in i_W'(X)&\iff h(i_U([\text{id}]_U))\in h(i_W(X))\\ &\iff i_U([\text{id}]_U)\in i_W(X)\end{align*}
Therefore by \cref{Reciprocity1}, the uniform \(M_W\)-ultrafilter derived from \(i_W'\) using \(i'_U([\textnormal{id}]_U)\) is equal to \(\tr{W}{U}\).
\end{proof}

As a corollary we can prove some basic facts about translation functions that are not at all obvious from the definition.

\begin{prp}[UA]\label{OrderPreserving}
For any countably complete ultrafilter \(W\), the function \(t_W : (\Un,\swo)\to (\Un^{M_W},\swo^{M_W})\) is order preserving.
\begin{proof}
Suppose \(U_0\swo U_1\). We must show \(\tr{W}{U_0}\swo^{M_W}\tr{W}{U_1}\). 

Let \((k_0,i_0) : (M_{U_0},M_W)\to N_0\) be the canonical comparison of \((j_{U_0},j_W)\) and let \((k_1,i_1) : (M_{U_1},M_W)\to N_1\) be the canonical comparison of \((j_{U_1},j_W)\).  Thus by \cref{Reciprocity1}, \(i_0 = j_{\tr{W}{U_0}}^{M_W}\) and \(i_1 = j_{\tr{W}{U_1}}^{M_W}\). Moreover by \cref{Reciprocity1}, letting \(\alpha_0 = [\text{id}]^{M_W}_{\tr{W}{U_0}}\) and \(\alpha_1 = [\text{id}]^{M_W}_{\tr{W}{U_1}}\), we have \(\alpha_0 = k_0([\text{id}]_{U_0})\) and \(\alpha_1 = k_1([\text{id}]_{U_1})\). 

Working in \(M_W\), fix a comparison \((h_0,h_1) : (N_0,N_1)\to P\) of \((i_0,i_1)\). To show show \(\tr{W}{U_0}\swo^{M_W}\tr{W}{U_1}\), it suffices to show that \(h_0(\alpha_0) < h_1(\alpha_1)\). But \(h_0(\alpha_0) = h_0\circ k_0([\text{id}]_{U_0})\) and \(h_1(\alpha_1) = h_1\circ k_1([\text{id}]_{U_1})\). Since \((h_0\circ k_0,h_1\circ k_1) : (M_{U_0},M_{U_1})\to P\) is a comparison of \((j_{U_0},j_{U_1})\), it must witness \(U_0\swo U_1\). That is, \(h_0\circ k_0([\text{id}]_{U_0}) < h_1\circ k_1([\text{id}]_{U_1})\). Hence \(h_0(\alpha_0) < h_1(\alpha_1)\), as desired.
\end{proof}
\end{prp}

\subsection{The linearity of the Ketonen order}
We now prove that the linearity of the Ketonen order implies the Ultrapower Axiom. The strategy is to prove the Reciprocity Lemma assuming only the linearity of the Ketonen order. Note first that translation functions can be defined assuming only the linearity of the Ketonen order.

\begin{defn}
Assume the Ketonen order is linear. If \(U\) is a countably complete ultrafilter and \(W\) is a countably complete uniform ultrafilter, then \(\tr{U}{W}\) denotes the \(\sE^{M_U}\)-least \(W_*\in \textnormal{Un}^{M_U}\) such that \(U^-(W_*) = W\).
\end{defn}

It is convenient to define an operation \(\oplus\) with the property that for any \(U\in \textnormal{Un}\) and \(W\in \textnormal{Un}^{M_W}\), \(U\oplus W\in \textnormal{Un}\) and \(j_{U\oplus W} = j^{M_U}_W\circ j_U\). (The usual ultrafilter sum operation does not have range contained in \(\textnormal{Un}\).) There are various ways in which one could do this, and our choice is motivated mostly by the desire that this operation work smoothly with the Ketonen order; see for example \cref{SumSeed}.

\begin{defn}
For \(\alpha,\beta\in \text{Ord}\), \(\alpha\oplus \beta\) denotes the {\it natural sum} of \(\alpha\) and \(\beta\), which is obtained as follows:

First write \(\alpha\) and \(\beta\) in Cantor normal form: \begin{align*}\alpha &= \sum_{\xi\in \text{Ord}} \omega^{\xi}\cdot m_\xi\\ \beta &= \sum_{\xi\in \text{Ord}} \omega^{\xi}\cdot n_\xi\end{align*} where \(m_\xi,n_\xi < \omega\) are equal to \(0\) for all but finitely many \(\xi \in \text{Ord}\). Then 
\begin{align*}\alpha \oplus \beta &= \sum_{\xi\in \text{Ord}} \omega^{\xi}\cdot (m_\xi + n_\xi)\end{align*}
\end{defn}

In other words one adds the Cantor normal forms of \(\alpha\) and \(\beta\) as polynomials.

The fact that natural addition is commutative and associative follows easily from the corresponding facts for addition of natural numbers. We mostly need the following triviality:
\begin{lma}\label{OrderSum}
If \(\alpha_0< \alpha_1\) and \(\beta\) are ordinals, then \(\alpha_0 \oplus \beta < \alpha_1 \oplus \beta\).
\end{lma}

\begin{defn}
If \(U\in \textnormal{Un}\) and \(W_*\in \textnormal{Un}^{M_U}\) then the {\it natural sum of \(U\) and \(W_*\)}, denoted \(U\oplus W_*\), is the uniform ultrafilter derived from \(j^{M_U}_{W_*}\circ j_U\) using \([\text{id}]^{M_U}_{W_*}\oplus j^{M_U}_{W_*}([\text{id}]_U)\). 
\end{defn}

The next lemma says that the natural sum of ultrafilters is Rudin-Keisler equivalent to the usual sum of ultrafilters.

\begin{lma}\label{UltrafilterSum} 
For any \(U\in \textnormal{Un}\) and \(W_*\in \textnormal{Un}^{M_U}\), \(M_{U\oplus W_*} = M^{M_U}_{W_*}\) and \(j_{U\oplus W_*} = j^{M_U}_{W_*}\circ j_U\).
\end{lma}

Natural sums also interact quite simply with the Ketonen order:

\begin{lma}\label{SumSeed}
Suppose \(U\in \textnormal{Un}\). Suppose \(W_0,W_1\in \textnormal{Un}^{M_U}\). Then \(W_0\sE^{M_U} W_1\) if and only if \(U\oplus W_0 \sE U\oplus W_1\).
\end{lma}

\begin{thm}[Reciprocity Theorem]\label{Reciprocity}
Assume the Ketonen order is linear. Then for any uniform countably complete ultrafilters \(U\) and \(W\), \[U\oplus \tr{U}{W} = W\oplus \tr{W}{U}\]
\begin{proof}
Assume towards a contradiction that \(U\oplus \tr{U}{W} \sE W \oplus \tr{W}{U}\). 

By \cref{SeedComparison}, there is a semicomparison \[(k,i) : (M_{U\oplus \tr{U}{W}},M_{W \oplus \tr{W}{U}})\to N\] of \((j_{U\oplus \tr{U}{W}}, j_{W \oplus \tr{W}{U}})\) such that \[k\left([\text{id}]_{U\oplus \tr{U}{W}}\right) < i\left([\text{id}]_{W \oplus \tr{W}{U}}\right)\]
In other words,
\begin{equation}\label{WrongOrder}k\left([\text{id}]^{M_U}_{\tr{U}{W}}\oplus j^{M_U}_{\tr{U}{W}}([\text{id}]_U)\right) < i\left([\text{id}]^{M_W}_{\tr{W}{U}}\oplus j^{M_W}_{\tr{W}{U}}([\text{id}]_W)\right)\end{equation}
\begin{clm}\label{MinimalityTranslates}
\(i\left([\textnormal{id}]^{M_W}_{\tr{W}{U}}\right) \leq k\left(j^{M_U}_{\tr{U}{W}}([\textnormal{id}]_U)\right) \)
\end{clm}
\begin{clm}\label{MinimalityDefinable}
\(i\left(j^{M_W}_{\tr{W}{U}}([\textnormal{id}]_W)\right) \leq k\left([\textnormal{id}]^{M_U}_{\tr{U}{W}}\right)\)
\end{clm}
Using \cref{OrderSum}, these two claims contradict \cref{WrongOrder}, so the assumption that \(U\oplus \tr{U}{W} \sE W \oplus \tr{W}{U}\) was false.
\begin{proof}[Proof of \cref{MinimalityTranslates}]
Let \(U_*\) be the \(M_W\)-ultrafilter derived from \(i\circ j^{M_W}_{\tr{W}{U}}\) using \(k\left(j^{M_U}_{\tr{U}{W}}([\textnormal{id}]_U)\right)\). Let \(h : M^{M_W}_{U_*}\to N\) be the factor embedding. Note that \(W^-(U_*) = U\): this is an easy calculation using \cref{LimitEmbedding}, noting that there is an elementary embedding \(M_U\to M^{M_W}_{U_*}\) witnessing the hypotheses of \cref{LimitEmbedding}, namely \(h^{-1}\circ k\circ j^{M_U}_{\tr{U}{W}}\).

If \(h([\text{id}]^{M_W}_{U_*}) < i([\text{id}]^{M_W}_{\tr{W}{U}})\), then \(U_*\sE^{M_W} \tr{W}{U}\) by \cref{SeedComparison}, contrary to the minimality of \(\tr{W}{U}\). Thus \[i([\text{id}]^{M_W}_{\tr{W}{U}})\leq  h([\text{id}]^{M_W}_{U_*}) = k(j^{M_U}_{\tr{U}{W}}([\textnormal{id}]_U))\] as desired.
\end{proof}
\begin{proof}[Proof of \cref{MinimalityDefinable}]
Let \(h : M_W\to M^{M_U}_{\tr{U}{W}}\) be the elementary embedding given by \cref{LimitEmbedding}. Then \(k\circ h([\text{id}]_W) = k([\textnormal{id}]^{M_U}_{\tr{U}{W}})\). Since \(i\circ  j^{M_W}_{\tr{W}{U}}\) and \(k\circ h\) are elementary embeddings \(M_W\to N\), and since \(i\circ  j^{M_W}_{\tr{W}{U}}\) is definable from parameters over \(M_W\), \[i\circ j^{M_W}_{\tr{W}{U}}\restriction \text{Ord} \leq k\circ h\restriction \text{Ord}\] by \cref{MinDefEmb}. In particular, \[i\circ j^{M_W}_{\tr{W}{U}}([\text{id}]_W)\leq k\circ h([\text{id}]_W) = k([\textnormal{id}]^{M_U}_{\tr{U}{W}})\qedhere\]
\end{proof}
Similarly we cannot have \(W \oplus \tr{W}{U}\sE U\oplus \tr{U}{W}\). By the linearity of the seed order, \(U\oplus \tr{U}{W} = W\oplus \tr{W}{U}\), finishing the proof of the Reciprocity Theorem.
\end{proof}
\end{thm}
\begin{proof}[Proof of \cref{Linearity}]
That (1) implies (2) is immediate.

We now prove (2) implies (1). By \cref{Reciprocity}, the linearity of the Ketonen order implies that for any uniform countably complete ultrafilters \(U\) and \(W\), \(U\oplus \tr{U}{W} = W\oplus \tr{W}{U}\). Thus \((j_{\tr U W}^{M_U},j_{\tr W U}^{M_W})\) is  a comparison of \((j_U,j_W)\). This implies that UA holds.
\end{proof}
\section{Ultrafilter Determinacy}
In this last section, we define a generalization of the Lipschitz order and raise the question of whether the Ultrapower Axiom is a long determinacy principle.

\begin{defn}
If \(\delta\) is an ordinal, a function \(f : P(\delta)\to P(\delta)\) is said to be {\it strongly Lipschitz} if for all \(X\subseteq \delta\), \(\alpha\in f(X)\) if and only if \(\alpha\in f(X\cap \alpha)\)\end{defn}

In other words, whether \(\alpha\) belongs to \(f(X)\) depends only on \(X\cap \alpha\). Recall that \(P(\text{Ord})\) denotes the class of sets of ordinals.

\begin{defn}
If \(\delta\) is an ordinal, the {\it Lipschitz order on \(P(P(\delta))\)} is defined on sets \(A,B\subseteq P(\delta)\) by setting \(A\sLi B\) if there is a strongly Lipschitz function \(f :P(\delta)\to P(\delta)\) such that \(A = f^{-1}[B]\). 
\end{defn}

The Lipschitz order is really a preorder.
Note that if \(U\) is an ultrafilter on \(\delta\), then \(U\subseteq P(\delta)\). The Lipschitz order on \(P(P(\delta))\) therefore induces a preorder on the set of ultrafilters on \(\delta\). 
\begin{lma}
For any ordinal \(\delta\), the Lipschitz order on \(P(P(\delta))\) restricts to a strict partial order on ultrafilters on \(\delta\).
\begin{proof}
Transitivity follows from the closure of strongly Lipschitz functions under composition. 

To show that \(\sLi\) is irreflexive, we use the following fact: if \(f : P(\delta)\to P(\delta)\) is a strongly Lipschitz function, then there is some \(A\subseteq \delta\) such that \(f(A) = \delta\setminus A\). (One defines such a set \(A\) by transfinite induction, putting \(\xi\in A\) if \(\xi\notin f(A\cap \xi)\).) Assume towards a contradiction that \(U\) is an ultrafilter on \(\delta\) and \(U\sLi U\). Let \(f : P(\delta)\to P(\delta)\) be a strongly Lipschitz function such that \(f^{-1}[U] = U\). Fix \(A\) such that \(f(A) = \delta\setminus A\). Then 
\begin{align*}
A\in U&\iff \delta\setminus A\notin U\\
&\iff f(A)\notin U\\
&\iff A\notin f^{-1}[U]\\
&\iff A\notin U
\end{align*}
This is a contradiction.
\end{proof}
\end{lma}

We would like to piece together the Lipschitz orders on \(P(P(\delta))\) for various ordinals \(\delta\) into a single order on all uniform ultrafilters. Since uniform ultrafilters on ordinals below \(\delta\) belong to \(P(P(\delta))\), it is tempting to define the Lipschitz order on uniform ultrafilters as the restriction of the Lipschitz order to uniform ultrafilters, but this would cause some minor problems when \(\delta\) is a successor ordinal. Instead w do the following:

\begin{defn}
The {\it Lipschitz order on \(\Un\)} is defined on countably complete uniform ultrafilters \(U\) and \(W\) by setting \(U\sLi W\) if \(U' \sLi W'\) in the Lipschitz order on \(P(P(\delta))\) where \(\delta = \max \{\textsc{sp}(U),\textsc{sp}(W)\}\) and \(U'\) and \(W'\) are the ultrafilters on \(\delta\) given by \(U\) and \(W\).
\end{defn}

\begin{prp}\label{LKet}
The Lipschitz order  on \(\Un\) extends the Ketonen order.
\begin{proof}
Suppose \(U\) and \(W\) are countably complete uniform ultrafilters with \(U\sE W\). We will show that \(U\sLi W\) in the Lipschitz order on \(\Un\). 

Let \(\delta = \max \{\textsc{sp}(U),\textsc{sp}(W)\} = \textsc{sp}(W)\) (by \cref{SpaceLemma1}) and let \(U'\) be the ultrafilter on \(\delta\) induced by \(U\). Fix a sequence of countably complete ultrafilters \(U_\alpha\) on \(\alpha\), defined for \(\alpha < \delta\), such that for all \(X\subseteq \textsc{sp}(U)\), \(X\in U\) if and only if \(X\cap \alpha\in U_\alpha\) for \(W\)-almost all \(\alpha < \delta\). Then for all \(X\subseteq \delta\), \(X\in U'\) if and only if \(X\cap \alpha\in U_\alpha\) for \(W\)-almost all \(\alpha < \delta\). 

Let \(f : P(\delta)\to P(\delta)\) be defined on \(X\subseteq \delta\) by \(f(X) = \{\alpha < \delta : X\cap \alpha\in U_\alpha\}\). Then \(f\) is strongly Lipschitz and  for any \(X\subseteq \delta\),
\begin{align*}
X\in U'&\iff \{\alpha < \delta : X\cap \alpha\in U_\alpha\}\in W\\
&\iff f(X)\in W\\
&\iff X\in f^{-1}[W]
\end{align*}
Therefore \(U'\sLi W\) in the Lipschitz order on \(P(P(\delta))\). It follows that \(U\sLi W\) in the Lipschitz order on \(\Un\).
\end{proof}
\end{prp}

A consequence of the proof of \cref{LKet} is that the Lipschitz order and the Mitchell order coincide on normal ultrafilters. 

Notice that for \(A,B\in P(P(\delta))\), \(A\sLi B\) holds if and only if Player I has a winning strategy in the following game of length \(2\cdot \delta\): I and II alternate to play out the characteristic functions of sets \(x,y\subseteq \delta\) with I playing at limit stages, and I wins if \(x\in A\) if and only if \(y\in B\). 
\begin{defn}
{\it Ultrafilter Determinacy} is the statement that the Lipschitz order on \(\Un\) is linear.
\end{defn}

\begin{thm}[UA]
Ultrafilter Determinacy holds, and in fact, the Lipschitz order on \(\Un\) is equal to the seed order.
\begin{proof}
The Lipschitz order extends the seed order by \cref{KSeed} and \cref{LKet}. Since the seed order is linear, the Lipschitz order is equal to the seed order.
\end{proof}
\end{thm}

It is not clear whether one can prove the wellfoundedness of the Lipschitz order on \(\Un\) in ZFC. On the other hand, using the Martin's proof of the wellfoundedness of the Wadge order, one can establish the following fact:

\begin{prp}[AD + DC]
The Lipschitz order on \(\Un\) is wellfounded.\qed
\end{prp}

\begin{qst}
Does Ultrafilter Determinacy imply the Ultrapower Axiom?
\end{qst}

\begin{qst}
Does the Axiom of Determinacy imply Ultrafilter Determinacy?
\end{qst}

We do not even know a counterexample to the semilinearity of the generalized Lipschitz order assuming AD.
\begin{defn}
Suppose \(\delta\) is an ordinal. A function \(f : P(\delta)\to P(\delta)\) is Lipschitz if for all \(\alpha < \delta\) and all \(A\subseteq \delta\), \(\alpha\in f(A)\) if and only if \(\alpha\in f(A\cap (\alpha+1))\).

The {\it nonstrict Lipschitz order on \(P(P(\delta))\)} is defined on \(A,B\subseteq P(\delta)\) by setting \(A\Li B\)  if there is a Lipschitz function \(f : P(\delta)\to P(\delta)\) such that \(A = f^{-1}[B]\).
\end{defn}

Under \(\text{AD}_\mathbb R\), the Lipschitz order is semilinear on \(P(P(\delta))\) for any \(\delta < \omega_1\). The following is therefore a natural question:

\begin{qst}[AD]
Are there \(A,B\subseteq P(\omega_1)\) with \(A\not \sLi B\) and \(B\not \Li P(\omega_1)\setminus A\)? 
\end{qst}

\bibliography{Bibliography}{}
\bibliographystyle{unsrt}

\end{document}